\documentclass[11pt]{article}
\usepackage{lmodern}
\usepackage{dsfont}
\usepackage{environ}
\input{definitions.sty}

\newcommand{\Th}{\mathrm{\Theta}}
\newcommand{\Omeg}{\mathrm{\Omega}}
\newcommand{\Omic}{\mathrm{O}}

\newcommand{\pmin}{{p_\text{min}}}
\newcommand{\pmax}{{p_\text{max}}}
\newcommand{\fmin}{{f_\text{min}}}
\newcommand{\fmax}{{f_\text{max}}}

\newcommand{\N}{{\mathbb{N}}}

\newcommand{\Rp}{{\mathbb{R}^+}}

\newcommand{\cone}{\text{cone}}

\newcommand{\eps}{\epsilon}

\newcommand{\calF}{\mathcal{F}}
\newcommand{\calG}{\mathcal{G}}

\newcommand{\calS}{\mathcal{S}}

\def\<{\langle}
\def\>{\rangle}

\newcommand{\emailhref}[1]{\href{mailto:#1}{\tt #1}}

\begin{document}

\title{Complexity-Approximation Trade-offs in Exchange Mechanisms: AMMs vs. LOBs}

\author{
 	    \textbf{Jason Milionis} \\
        \small Department of Computer Science \\
 		\small Columbia University \\
 		\small \emailhref{jm@cs.columbia.edu}
 		\and
 		\textbf{Ciamac C. Moallemi} \\
        \small Graduate School of Business \\
 		\small Columbia University \\
 		\small \emailhref{ciamac@gsb.columbia.edu}\\
 		\and
 		\textbf{Tim Roughgarden} \\
        \small Department of Computer Science \\
 		\small Columbia University \\ \small a16z Crypto \\
 		\small \emailhref{tim.roughgarden@gmail.com}
}
\date{Initial version: September 26, 2022 \\
      Current version: April 19, 2023
}
\maketitle

\thispagestyle{empty}

\begin{abstract}
This paper presents a general framework for the design and analysis of exchange mechanisms between two assets that unifies and enables comparisons between the two dominant paradigms for exchange, constant function market markers (CFMMs) and limit order books (LOBs). In our framework, each liquidity provider (LP) submits to the exchange a downward-sloping demand curve, specifying the quantity of the risky asset it wishes to hold at each price; the exchange buys and sells the risky asset so as to satisfy the aggregate submitted demand. In general, such a mechanism is budget-balanced (i.e., it stays solvent and does not make or lose money) and enables price discovery (i.e., arbitrageurs are incentivized to trade until the exchange's price matches the external market price of the risky asset). Different exchange mechanisms correspond to different restrictions on the set of acceptable demand curves.

The primary goal of this paper is to formalize an approximation-complexity trade-off that pervades the design of exchange mechanisms. For example, CFMMs give up expressiveness in favor of simplicity: the aggregate demand curve of the LPs can be described using constant space (the liquidity parameter), but most demand curves cannot be well approximated by any function in the corresponding single-dimensional family. LOBs, intuitively, make the opposite trade-off: any downward-slowing demand curve can be well approximated by a collection of limit orders, but the space needed to describe the state of a LOB can be large.

This paper introduces a general measure of {\em exchange complexity}, defined by the minimal set of basis functions that generate, through their conical hull, all of the demand functions allowed by an exchange. With this complexity measure in place, we investigate the design of {\em optimally expressive} exchange mechanisms, meaning the lowest complexity mechanisms that allow for arbitrary downward-sloping demand curves to be approximated to within a given level of precision. Our results quantify the fundamental trade-off between simplicity and expressivity in exchange mechanisms.

As a case study, we interpret the complexity-approximation trade-offs in the widely-used Uniswap v3 AMM through the lens of our framework.
\end{abstract}

\section{Introduction}

\newcommand{\nume}{num\'eraire}

Decentralized exchanges are now an integral part of the broader ecosystem of blockchains, as
evidenced by their ever growing volume of transactions \parencite{kaiko_crypto_2022}.
On model centralized exchanges,
the exchange of a risky asset for a \nume is typically carried out by an exchange mechanism
known as an electronic limit order book (LOB), in which market participants specify quantities of
shares of the risky asset they would like to trade at specified prices. Trades then occur as
orders are matched in a greedy way: whenever there is overlap between bid and ask prices (i.e.,
between a buy and a sell), a trade is executed, and the matched orders are cleared from the LOB.
LOBs therefore maintain and update a list of all the currently outstanding buy and sell orders.

LOBs face two types of challenges in an decentralized environment such as the Ethereum blockchain.
First, because storage and computation in such an environment tend to be so scarce, implementing
an LOB can be prohibitively expensive.  Second, LOBs are well known suffer from liquidity problems
in thin markets (markets with few buyers or sellers), for example, for ``long-tail'' crypto assets.

These challenges have motivated
an alternative exchange design that has become very widely used in
blockchains:
automated market makers (AMMs) and, in particular, constant
function market makers (CFMMs).  Uniswap
\parencite{adams2021uniswap,adams2020uniswap} is the most well known and
widely used example of a CFMM.

AMMs address the second challenge above by
offering guaranteed liquidity, meaning at all times
there is a spot price between $0$ and $\infty$ at which the AMM is
willing to buy or sell.
AMMs like Uniswap address the first challenge by
using only simple calculations and data structures. For example,
for the canonical (``$xy=k$'') constant product market maker,
the state of mechanism can be described by two numbers (the
quantities~$x$ and~$y$ held by the pool), and there is a simple
closed-form formula (requiring only a small number of additions,
multiplications, divisions, and square roots) for computing the quantity of the
risky asset received in exchange for a specified amount of the num\'eraire
(as a function of $x$ and $y$).

In this paper, we provide a general framework for describing and
reasoning about exchange mechanisms, which enables
``apples-to-apples'' comparisons between LOBs and AMMs on metrics such
as complexity and expressiveness.
More specifically, our contributions can be delineated as follows:
\begin{enumerate}
\item We provide a \textbf{common framework} for describing exchange mechanisms that encompasses
  both CFMMs and LOBs. In our general model, liquidity providers (LPs) submit to the exchange
  their preferences (in the form of what we define as \textbf{demand curves} for the risky asset)
  along with appropriate deposits of the risky asset and num\'eraire (see \Cref{sec:model} for details).
\item We formalize the sense in which some methods of exchange are simpler than others,
  introducing a general notion of \textbf{exchange complexity}.  Exchange complexity is defined by
  the minimal set of basis functions that generate, through their conical hull, all of the demand
  functions allowed by an exchange.  We classify the complexity of all the prominent types of
  exchange mechanisms
(see \Cref{sec:complex_n_ex} for details).
\item We characterize the \textbf{fundamental trade-off} between
  the \emph{complexity} of an exchange (in a sense that we define) %
  and %
  the \emph{expressibility} of an exchange as measured by its ability
  to approximate arbitrary preferences of the LPs (i.e., arbitrary
  demand curves). In particular, we prove matching (up to constant
  factors) upper and lower
  bounds
  on the minimum exchange complexity necessary to attain
  a specified approximation error (see \Cref{sec:tradeoffs} for details).
\item As a case study, we interpret the complexity-approximation trade-offs in the widely-used Uniswap v3 AMM through the lens of our framework
  (see \Cref{sec:v3} for details).
\end{enumerate}

\subsection{Literature Review}
\label{subsec:litrev}

The use of AMMs for decentralized exchange mechanisms was first
proposed by \textcite{vbuterin_lets_2016} and \textcite{lu_building_2017}.
The latter authors suggested a constant product market maker, which
was first analyzed by \textcite{angeris2019analysis}.
\textcite{angeris2021replicatingmarketmakers,angeris2021replicatingmonotonicpayoffs}
define and use a reparameterization of a CFMM curve (established by
\textcite{angeris2020improved}) in terms of portfolio holdings of the
pool with respect to the price as a tool to replicate payoffs and
compute the pool's value function; we use this same reparameterization
for different purposes, to define a general (i.e.,  not AMM-specific)
framework of exchange
and identify fundamental complexity-approximation trade-offs in
exchange design.

A separate line of work seeks to design specific CFMMs with good
properties by identifying good bonding functions, %
variations and
combinations of CFMMs in a dynamic setting with a specific focus
on optimizing fees, and minimizing arbitrage and slippage
\parencite{angeris2022optimalrouting,engel2021composing,engel2021presentation,port2022mixing,wu2022constant,ciampi2022fairmm,forgy2021family,krishnamachari2021dynAMMs,krishnamachari2,lambdaCFMMs,jensen2021homogenous}.
While fees could be easily integrated into our model, they
have no bearing on complexity-approximation trade-offs and thus we
generally ignore them in this paper for simplicity.

Some previous papers propose generalizations of CFMMs to somewhat
wider classes of exchanges~\parencite{bichuch2022axioms,xu2021sok} without
considering LOBs.

CFMMs and LOBs have been compared before (in ways orthogonal to the
questions studied here)
\parencite{barbon2021quality,lehar2021decentralized,capponi2021adoption}.
Most of these works either compare
the observed liquidities and the price efficiency of these mechanisms
\parencite{lehar2021decentralized,capponi2021adoption} or study the same
through the lens of arbitrage bounds \parencite{barbon2021quality}.
\textcite{young2020equivalence} argues that AMMs can be interpreted as
``smooth order books'' and notes a type of non-uniform converse (with each
possible state of a smooth order book represented using a different AMM).
\textcite{chitra2021liveness} compare CFMMs and LOBs in terms of the
number of arbitrage transactions necessary to recover from a liveness
attack on the underlying blockchain.

Another line of work analyzes competition between CFMMs and LOBs and
the consequent liquidity properties of both at equilibrium
\parencite{aoyagi2020liquidity,aoyagi2021coexisting,capponi2021adoption}.
\textcite{goyal2022batch} consider the computational
complexity
of computing such equilibria.

There is a large literature on the market microstructure of limit
order books; see the textbook by \textcite{ohara} and references therein.
There are some examples of on-chain LOBs on high-throughput blockchains
\parencite{solanaLOB,moosavi2021lissy}.

Finally, \textcite{adams2021uniswap} suggest that Uniswap v3's key feature
is that ``LPs can approximate any desired distribution of liquidity on
the price space,'' with empirical backing provided by
\textcite{huynh2022providing}; one application of our work is to put this
intuition on sound mathematical footing.
There is also work on Uniswap v3 from the LP perspective, such as how
beliefs about future prices should guide the choice of an LP's demand
curve
\parencite{fan2022differential,yin2021liquidity,neuder2021strategic}.

\section{Model}
\label{sec:model}

\subsection{Model Primitives}
\label{subsec:model_prim}

We begin by describing our framework for exchange design.
While this paper uses this framework specifically to study fundamental
complexity-approximation trade-offs in exchange mechanisms, we believe
it can serve also as a starting point for many future investigations.

Suppose there are two assets, a risky asset and a num\'eraire asset.  Each LP comes separately to
the exchange, and declares the amount of risky asset they would like to hold at each possible
price $p$, i.e., a non-increasing, non-negative function $g_i \colon (0,\infty) \to \Rp$. We
call the function $g_i(\cdot)$ the $i$th LP's \textbf{demand curve} for the risky asset, because
it refers to the demand of the LP for the risky asset (i.e., we are considering the perspective of
the LP).  Assuming that the current price is $p_0$, the LP simultaneously deposits a quantity
$g_i(p_0)$ of the risky asset in the common pool, along with an amount of num\'eraire given by the
Riemann–Stieltjes integral
\begin{equation}
\label{eq:lp_numeraire_deposit}
-\int_0^{p_0} p \, dg_i(p)
\,.
\end{equation}
Note that this integral is well-defined (though possibly infinite) since $g_i(\cdot)$ is
monotonic. Moreover, the integral is non-negative since $g_i(\cdot)$ is non-increasing.  In cases
where $g_i(p)$ is differentiable, the differential takes the form $dg_i(p) = g_i'(p) \, dp$.  We
will show later that this deposit of num\'eraire is necessary and sufficient for the exchange to
be budget-balanced or solvent, i.e., the exchange system does not extend credit.

The exchange mechanism maintains the demand curves of the LPs, along with the current price $p_0$.
Assuming that $n$ liquidity providers have contributed to the exchange their demand curves along
with respective payments of risky asset and num\'eraire, the \emph{aggregate demand curve} (i.e.,
the total quantity of risky asset that the exchange will hold at any given price) is given by the
non-increasing function
\begin{equation}
\label{eq:g}
g(p) = \sum_{i=1}^n g_i(p)
\,.
\end{equation}
Addition and removal of liquidity (LP ``mints'' and ``burns'', as they are known in
practice) simply occur through additions and removals of particular $g_i$'s to the aggregate
demand curve of the exchange.  These demand curves of the LPs can arise through
\emph{bonding curves} of traditional CFMMs (i.e., functions $f$ such that the holdings of the
joint pool $(x,y)$ satisfy $f(x,y)=c$ for some $c$) but this is not necessary; i.e., the exchange
mechanisms defined by our framework strictly generalize AMMs.

\paragraph{Trading} A liquidity demanding trader who wants to trade with the exchange will do so
by specifying a target (new) price $p_1\ne p_0$. The trader gets a quantity $g(p_0) - g(p_1)$ of
risky asset, and pays the following amount in num\'eraire:
\begin{equation}
\label{eq:trade_numeraire}
-\int_{p_0}^{p_1} p \, dg(p)
\,,
\end{equation}
as determined by the aggregate liquidity of the exchange $g(p)$ of \Cref{eq:g}. As was the case
for \Cref{eq:lp_numeraire_deposit}, this integral is well-defined, it is non-negative if $p_1 \geq
p_0$, and non-positive if $p_1 \leq p_0$.

\paragraph{Uniswap v2 example} To give a simple example, the particular case of a constant product
market maker (CPMM), such as Uniswap v2, arises from our mechanism as follows: restrict the set of
allowable demand curves $g_i$ that an LP may submit to the form
\[
g_i(p)=\frac{c_i}{\sqrt{p}}
\,,
\]
for some constant $c_i>0$. Then, the aggregate demand curve of the exchange will be of the form
\[
g(p)=\sum_{i=1}^n g_i(p)=\frac{c}{\sqrt{p}}
\,,
\]
for $c=\sum_{i=1}^n c_i > 0$. A trader who will trade with this exchange at a current price $p_0$ with a target price $p_1$ (or equivalently, with a specific quantity of risky asset to be purchased, since there a one-to-one correspondence) will obtain a quantity $g(p_0) - g(p_1) = c\left( \frac{1}{\sqrt{p_0}} - \frac{1}{\sqrt{p_1}} \right)$ of risky asset, and pay in num\'eraire
\[
- \int_{p_0}^{p_1} pg'(p)\, dp
= \int_{p_0}^{p_1} \frac{c}{2\sqrt{p}} \, dp
= c \left( \sqrt{p_1} - \sqrt{p_0} \right)
\,.
\]
Comparing this to the same expressions for an ``$xy=k$'' CPMM, the
trader gets exactly the same quantity of risky asset and pays exactly
the same amount of num\'eraire as they would in the ``$xy=k$'' CPMM,
with $k=c^2$.  Essentially, the curve $g(p)$ above is just a
reparameterization of the CPMM curve $xy=k$ in terms of prices
\parencite{angeris2021replicatingmonotonicpayoffs} where the risky asset
is available in quantity $x$ in the pool and the amount of num\'eraire
is $y$\footnote{In particular, $x=g(p)=c/\sqrt{p}$ and $y=c\sqrt{p}$ at all times in the pool for the corresponding defined price $p$.}.

\paragraph{Significance of LPs' demand curves} In this mechanism, we
view the individual demand curves chosen by the LPs as their
\emph{ideal preferences} with respect to risky asset holdings at each
price in regards to their market making activity.
They are in some sense ``forced'' to make the
market ---this is tautologically the reason that they participate in
the exchange as LPs\footnote{Note that LPs may also hold other portfolios of the risky asset, which of course need not be restricted to be non-increasing in the asset price, but their individual demand curves when they participating in an exchange mechanism need to reflect exactly and only the activity of making the market.}--- but \emph{exactly how} they
do this is
specified by
the shape of their demand curves.  The
requirement that each $g_i$ be non-increasing can be explained through
this argument: each demand curve of any LP has to always correspond to
making the market; as the price of the risky asset increases, a market
maker may only decrease their holdings of the asset (i.e., sell the asset), because if at any given price their holdings as defined in the exchange mechanism marginally increased (i.e., the LP would buy the risky asset at the marginal price), then any trader would sweep such a marginal quantity as it is to their advantage.

\subsection{Price Discovery and Budget Balance}
\label{subsec:model_props}

In the previous section, we defined a framework for an exchange mechanism. In order for an
exchange to be reasonable, two properties would be necessary: (1) price discovery should occur,
i.e., given an outside market with a fixed external market price, the exchange's price should
eventually become identical to the market price; and (2) the exchange should at no point in time
become insolvent, i.e., any feasible trade should always keep the amount of num\'eraire
non-negative.  (Because demand curves are non-negative, the amount of the risky asset is
automatically non-negative.)  Equivalently, the second property is broadly known in financial
markets as a ``no credit'' requirement, i.e., that the exchange does not incorporate the ability
of LPs to take credit. In the remainder of the section, we formalize and prove these properties
for our model.

\begin{proposition}
[Price discovery]
If there exists an outside market with fixed external market price $p$ of the risky asset with respect to the num\'eraire, then external market participants (arbitrageurs) always have financial incentive to trade with an exchange defined as per the framework of \Cref{subsec:model_prim} until the price of such exchange becomes equal to the external market price.
\end{proposition}

\begin{proposition}
[Budget balance]
An exchange defined as in the framework of \Cref{subsec:model_prim} is budget-balanced or solvent, i.e., the amount of num\'eraire that the joint pool contains at all times (with any sequence of feasible trades, or liquidity additions/removals) is non-negative.
\end{proposition}

We defer the full proofs of these two propositions to \Cref{app:proofs_model_props}.

\section{Exchange Description Complexity \& Examples}
\label{sec:complex_n_ex}

Our general model in \Cref{subsec:model_prim} allows LPs to submit arbitrary
downward-sloping demand curves.  Such curves are not generally
representable in a finite amount of space, so practical considerations
suggest restricting the space of demand curves that LPs are allowed 
to submit.
We will say that an \emph{exchange mechanism} is a restriction of
the general exchange framework of \Cref{subsec:model_prim} in which
each LP demand curve is required to belong to a set of allowable
demand curves, i.e., $g_i \in \calG$ for some class $\calG$ of
non-increasing, non-negative functions over the positive reals.
An exchange mechanism, then, is defined by the choice of class~$\calG$.

Towards defining a measure of exchange complexity, we will be
interested in succinct ways of representing all the demand
functions~$g$ in a class~$\calG$.
Specifically, given an arbitrary such class~$\calG$, we can consider
its conical hull. This is the smallest convex cone that contains\footnote{This definition makes sense because the
  intersection of convex cones is again a convex cone; see,
  e.g., \textcite{rockafellar} for further background.}
$\calG$
or, equivalently, the closure of~$\calG$ under finite non-negative
linear combinations:
\[
\cone(\calG) = \left\{ \sum_{i=1}^k c_i g_i(p) : g_i(p)\in\calG, c_i \ge 0, k\in\N \right\}
\,.
\]
In our context, non-negative linear combinations can be interpreted as
aggregations of multiple LP positions.

A {\em basis} of a cone is a minimum-cardinality set of elements that
generates the cone, 
meaning a set $\calS$ such that $\cone(\calS) = \cone(\calG)$.
We then define 
the \textbf{exchange complexity} of an exchange (i.e., a choice
$\calG$ of allowable demand functions) as the cardinality of a basis
for  $\cone(\calG)$.\footnote{While our formalism in principle accommodates exchanges with infinite exchange complexity, any practical exchange needs to be defined by a finite basis on any compact (sub-)domain. Additionally, our results only make use of exchanges that have a finitely generated conic closure to approximate any demand curve within a finite approximation error under reasonable assumptions about the error metrics.}
By definition, if a set $\calG$ of demand functions has exchange
complexity~$k$, every function of~$\calG$ can be represented by a
$k$-tuple of non-negative real numbers (one coefficient for each of the
basis functions).\footnote{The focus of this work is on information-theoretic complexity -- approximation trade-offs, and we do not explicitly model computation. However, our positive results only make use of mechanisms for which computation with basis functions is straightforward.}

Our measure of exchange complexity is, by design, well defined for an
arbitrary collection $\calG$ of allowable demand functions.  In all
the real-world examples that we are aware of, this set $\calG$ is
already closed under non-negative linear combinations (i.e., is a
cone).  In this case, exchange complexity effectively counts an
exchange's ``primitive'' LP positions from which all possible
aggregations of LP positions can be derived.

This definition of exchange complexity allows us to formalize
the intuition that some exchanges are easier to represent than
others (e.g., that CFMMs are simpler than LOBs).
Next, we evaluate the exchange complexity of all of the most popular
types of exchanges used to trade crypto assets.

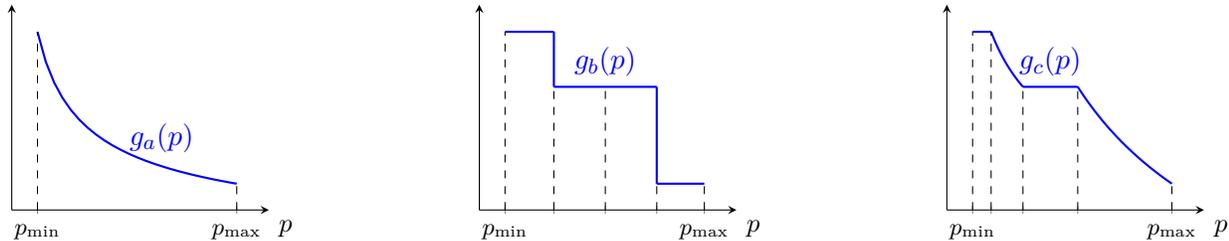
\begin{figure}
\centering
\captionsetup[subfigure]{labelformat=empty}
\begin{subfloat}[\label{subfig:cfmm}]
\centering
    \begin{tikzpicture}
      \begin{axis}[
        clip=true,
        axis lines=left,
        xtick={0.2, 1.75},
        xticklabels={$\pmin$, $\pmax$},
        ytick=\empty,
        yticklabels={},
        xlabel=$p$,
        small,
        xmin=0, xmax=2,
        ymin=0.5, ymax=2.5,
        footnotesize,
        ];
        \addplot[mark=none,blue,thick,domain=0.2:1.75]{1 / sqrt(x)} node[midway,right=7] {$g_a(p)$};
        \addplot[mark=none,dashed] coordinates {(0.2, 0.5) (0.2, 2.236)};
        \addplot[mark=none,dashed] coordinates {(1.75, 0.5) (1.75, 0.756)};
      \end{axis}
    \end{tikzpicture}
\end{subfloat}
\hfill
\begin{subfloat}[\label{subfig:lob}]
\centering
\begin{tikzpicture}
      \begin{axis}[
        clip=true,
        axis lines=left,
        xtick={0.2, 0.58, 0.98, 1.38, 1.75},
        xticklabels={$\pmin$, , , , $\pmax$},
        ytick=\empty,
        yticklabels={},
        xlabel=$p$,
        small,
        xmin=0, xmax=2,
        ymin=0.5, ymax=2.5,
        footnotesize,
        ];
        \addplot[mark=none,blue,thick,domain=0.2:0.58]{2.236};
        \addplot[mark=none,blue,thick] coordinates {(0.58, 1.7) (0.58, 2.236)};
        \addplot[mark=none,blue,thick,domain=0.58:1.38]{1.7} node[midway,above] {$g_b(p)$};
        \addplot[mark=none,blue,thick] coordinates {(1.38, 0.756) (1.38, 1.7)};
        \addplot[mark=none,blue,thick,domain=1.38:1.75]{0.756};
        \addplot[mark=none,dashed] coordinates {(0.2, 0.5) (0.2, 2.236)};
        \addplot[mark=none,dashed] coordinates {(0.58, 0.5) (0.58, 1.7)};
        \addplot[mark=none,dashed] coordinates {(0.98, 0.5) (0.98, 1.7)};
        \addplot[mark=none,dashed] coordinates {(1.38, 0.5) (1.38, 0.756)};
        \addplot[mark=none,dashed] coordinates {(1.75, 0.5) (1.75, 0.756)};
      \end{axis}
    \end{tikzpicture}
\end{subfloat}
\hfill
\begin{subfloat}[\label{subfig:v3}]
\centering
\begin{tikzpicture}
      \begin{axis}[
        clip=true,
        axis lines=left,
        xtick={0.2, 0.344, 0.5917, 1.0177, 1.75},
        xticklabels={$\pmin$, , , , $\pmax$},
        ytick=\empty,
        yticklabels={},
        xlabel=$p$,
        small,
        xmin=0, xmax=2,
        ymin=0.5, ymax=2.5,
        footnotesize,
        ];
        \addplot[mark=none,blue,thick,domain=0.2:0.344]{2.236};
        \addplot[mark=none,blue,thick,domain=0.344:0.5917]{1.7+(2.236-1.7) * (1 / sqrt(x) - 1 / sqrt(0.5917)) / (1 / sqrt(0.344) - 1 / sqrt(0.5917))};
        \addplot[mark=none,blue,thick,domain=0.5917:1.0177]{1.7} node[midway,above] {$g_c(p)$};
        \addplot[mark=none,blue,thick,domain=1.0177:1.75]{0.756+(1.7-0.756) * (1 / sqrt(x) - 1 / sqrt(1.75))
          / (1 / sqrt(1.0177) - 1 / sqrt(1.75))};
        \addplot[mark=none,dashed] coordinates {(0.2, 0.5) (0.2, 2.236)};
        \addplot[mark=none,dashed] coordinates {(0.344, 0.5) (0.344, 2.236)};
        \addplot[mark=none,dashed] coordinates {(0.5917, 0.5) (0.5917, 1.7)};
        \addplot[mark=none,dashed] coordinates {(1.0177, 0.5) (1.0177, 1.7)};
        \addplot[mark=none,dashed] coordinates {(1.75, 0.5) (1.75, 0.756)};
      \end{axis}
    \end{tikzpicture}
\end{subfloat}
\caption{$g\in\cone(\calG)$ for three typical cases: (a) CPMM, (b) LOB, (c) Uniswap v3}
\label{fig:ex}
\end{figure}

\paragraph{CFMMs} CFMMs are generated by the restriction to
non-negative scalar multiples of a \emph{single} basis function, i.e.,
$ \calG = \{ c\cdot g(p) : c\ge 0 \} , $ where $g(p)$ is \emph{one}
reference demand curve, out of all the possible curves of the CFMM.  The coefficient $c$ of
this basis function can then be interpreted as the liquidity
parameter.  As an example, for the CPMM, we can choose
$g(p)=1/\sqrt{p}$ (cf., \Cref{subfig:cfmm}); the coefficient can be
interpreted as $\sqrt{k}$ for the~$k$ in ``$xy=k$.''  In general,
irrespective of the bonding curve, the exchange complexity of a CFMM
is $1$.  Under standard assumptions (e.g., as in~\textcite{angeris2021replicatingmarketmakers}) on a
CFMM's bonding curve~$f$, the corresponding basis function~$g$ can be
derived from~$f$ in a mechanical way, through optimization.

\paragraph{LOBs} Limit order books consist of limit orders, which are
(buy or sell) orders of quantities of the  risky asset at some
price. The  predetermined prices at which limit orders can be
specified are called \emph{ticks}. In our framework, limit orders
can be represented by a set of basis functions in which each function
corresponds to a limit order at a specific tick (i.e., a step
function, where the step occurs at the tick).
According to our definition of exchange complexity above, then, the
exchange complexity of a limit order book (cf., \Cref{subfig:lob})
with $k$ ticks is $k$.
If we restrict our attention to a
price range $[\pmin, \pmax]$ with ticks $\pmin$, $\pmin + \epsilon$,
$\pmin + 2\epsilon$, \ldots, $\pmax$,
the exchange complexity of such a LOB would be
$
(\pmax-\pmin)/\eps
.
$

There is a superficial difference in convention between traditional
LOBs and our model of them in the preceding paragraph, concerning
the \emph{default action} after a trade that crosses the price of a
limit order.
In an LOB, the matching limit order would be automatically removed
from the order book, whereas in our
framework here the corresponding LP would, in effect, automatically 
place a new limit order in the opposite direction at the same price.
In other words, a LOB basis function
is equivalent to both a limit buy and a limit sell at the tick price,
and which one takes effect depends on the current price $p_0$ and the
trade to be executed. 
Because limit orders can be easily added to or removed from
traditional LOBs, and because our model accommodates LP mints and
burns, there is no material difference between the two viewpoints.

\paragraph{Uniswap v3} Uniswap v3 (cf., \Cref{subfig:v3}) can be
viewed as a hybrid of a CFMM and a LOB, with the CPMM curve applied
only within a short price interval (in between two of the pre-defined
ticks).  By
allowing multiple intervals, Uniswap v3  allows concentrated positions
in the spirit of LOBs, a property known as
\emph{concentrated liquidity}.
If there are $k$ ticks contained in the
interior of an interval $[\pmin, \pmax]$, then Uniswap v3's complexity
on this interval is $k$.  (There is one basis function for each price
segment $[t_i, t_{i+1}]$ between two successive ticks; the function is constant up
until the interval, decreases as in a CPMM within the interval, and
is zero after the interval, as in \Cref{eq:v3_basis}).
\begin{equation}
\label{eq:v3_basis}
g_i(p) =
\begin{cases}
    \frac{1}{\sqrt{t_i}} - \frac{1}{\sqrt{t_{i+1}}} \,, &\text{for } p\le t_i \\
    \frac{1}{\sqrt{p}} - \frac{1}{\sqrt{t_{i+1}}} \,, &\text{for } t_i \le p \le t_{i+1} \\
    0 \,, &\text{for } p \ge t_{i+1}
\end{cases}
\end{equation}

Thus, the exchange complexity of both LOBs
and Uniswap v3 is controlled by the number of ticks (independent of
the spacing between them).  In practice,
ticks are sparser in Uniswap v3 than in a traditional LOB, and
the former accordingly has lower exchange complexity than the latter.
For an example
calculation, if the ticks in Uniswap v3 are assumed to be of the form
$1.0001^i$, and $\pmin=1.0001^s, \pmax=1.0001^{s+t}$, then Uniswap
v3's complexity in the price interval $[\pmin,\pmax]$ is
\[
t=\frac{\log(\pmax/\pmin)}{\log 1.0001} \approx 10000.5 \log(\pmax/\pmin)
\,.
\]
We note that range orders in Uniswap v3 correspond to sums of
single-interval positions (with one position per interval in the
range) and are therefore automatically included in the cone generated
by the basis functions defined above.

\section{Complexity -- Approximation Trade-offs}
\label{sec:tradeoffs}

\subsection{Notions of Approximation}
\label{subsec:approx}

Having defined the complexity of an exchange mechanism, we turn to
defining the {\em expressiveness} of such a mechanism and proving
fundamental trade-offs between complexity and expressiveness.
Informally, we will measure the expressiveness of an exchange
mechanism via the extent to which its allowable demand curves (i.e.,
the functions in the class $\calG$) can represent arbitrary LP
preferences (i.e., an arbitrary demand curve).

Precisely, denote by $\calF$ the class of all non-increasing functions
$f:[\pmin, \pmax] \to [\fmin, \fmax]$. This is the most general class
of bounded demand curves according to our framework. Any arbitrary
(bounded) preference of an LP will be some specific non-increasing
function $f\in\calF$.\footnote{Note that in what follows $f$ is a
  demand curve, as defined in \Cref{subsec:model_prim}, and not a
  bonding curve of a CFMM.}
We next define the extent to which some allowable demand curve $g \in
\calG$ (with the same domain and range) approximates~$f$.
(In this section we use~$g$ rather than~$g_i$ to denote an arbitrary
function of~$\calG$.)%
\footnote{The restricting to a bounded domain and range is convenient
  but can be relaxed considerably.  The fundamental issue is that, to
  meaningfully speak about function approximations and avoid infinite
  distances between distinct functions, we need to impose
  constraints on allowable demand functions and/or the choice of
  distance function and underlying measure (on prices).
Functions with bounded domain and range are convenient because they
are integrable no matter what the distance notion and measure.  Our
results can be generalized by considering combinations of demand
function classes and classes of measures for which the same
integrability properties are guaranteed.

Additionally, it will be apparent from our lower bound
(\Cref{thm:lower}) that, if the family of functions $\calF$ was not
bounded by some finite bound $\fmax < \infty$, there would be no
finite approximation error guarantee with any finite complexity (under
any natural notion of approximation error).
}

First, we introduce the weighted $\ell_p$ norm in the function space
as a distance metric; without loss of generality, assume we have a
normalized (and integrable) weight function
$w : [\pmin,\pmax] \to \Rp$ such that $\int_\pmin^\pmax w(p) \, dp = 1$. Then, the
weighted $\ell_p$ distance of two functions $f,g\in\calF$ is
\[
d(f, g) = \left( \int_\pmin^\pmax w(s) \left| f(s)-g(s) \right|^p \, ds \right)^{1/p}
\,.
\]
The weight function~$w$ can be interpreted as a measure on the price
space, for example reflecting a belief (by an LP, the AMM designer, or
the community) that some prices may be more relevant than others.  On
a first read, we encourage the reader to take~$w$ to be the constant
function $w(s) = 1/(\pmax-\pmin)$ for all~$s \in [\pmin,\pmax]$.

Given this definition, we define the \textbf{approximation error} of
the exchange defined by $\calG$ as the worst-case (over arbitrary LP
preferences/demand curves $f\in\calF$) distance from the best-case
approximation (over allowable functions $g\in\cone(\calG)$) of $f$, as
above:
\begin{equation}
\label{eq:err}
\text{err}(\calG) = \sup_{f\in\calF} \left\{ \inf_{g\in\cone(\calG)} d(f,g) \right\}.
\end{equation}

\subsection{Upper and Lower Bounds}
\label{subsec:main_bounds}

From the AMM designer's perspective, an ``optimal'' AMM would
enable LPs to have their preferences
expressed closely; a bit more formally, the worst-case
approximation error through the AMM for arbitrary LP demand curves
should be low, and intuitively should decrease with the complexity of
the exchange mechanism: the higher exchange complexity should result
in a payoff of lower
worst-case approximation error.
The results below characterize this trade-off, by identifying the
best-possible worst-case approximation error as a function of the
exchange complexity.
For example, for the special case in which
the approximation metric between two functions
is the (unweighted) $\ell_1$ distance,
an exchange complexity (equivalently, number of
basis functions) of $\Th(1/\eps)$ is necessary and sufficient to
achieve an $\eps$ worst-case approximation error.

Our upper bound argument also implies the (intuitive but previously
unformalized) fact that limit order books %
at appropriately
defined  price ticks attain the optimal
approximation error guarantee for a given level of exchange complexity
(up to a factor of~2).
In other words, when computation and storage
are not first-order constraints, LOBs
are nearly optimally expressive exchange mechanisms.

\begin{theorem}
[Upper bound]
\label{thm:upper}
For every $\eps>0$, there exists a limit order book (LOB) exchange mechanism $\calG$ with exchange complexity $k = \Omic(1/\eps^p)$ that attains approximation error
\[
\text{err}(\calG) \le \eps \cdot \frac{\fmax - \fmin}{2}
\,.
\]
\end{theorem}

\begin{theorem}
[Lower bound]
\label{thm:lower}
For every $\eps>0$, every exchange mechanism $\calG$ with exchange complexity $\Omic(1/\eps^p)$ suffers approximation error
\[
\text{err}(\calG) \ge \eps \cdot \Omeg(\fmax - \fmin)
\,.
\]
\end{theorem}

For the detailed proofs of \Cref{thm:upper,thm:lower} we refer to \Cref{subsec:proofupper,subsec:prooflower} respectively.

\section{Uniswap v3}
\label{sec:v3}

Next, we answer the question: to what extent do various formats in practice come close to this complexity -- approximation trade-off?
Historically, constant product market makers (CPMMs) were first built for gas efficiency purposes \parencite{adams2020uniswap}, but when it was realized that this came often at the expense of capital efficiency, the proposal of Uniswap v3 came around \parencite{adams2021uniswap}, which trades like a CPMM curve inside tight intervals at a pre-defined tick spacing, which are otherwise independent.
In this section, we consider Uniswap v3, which is at the time of writing a widely used AMM, as an enlightening example to showcase how our theory can be applied to formally prove approximation guarantees for AMMs employed in practice.

More specifically, we can prove that ---under a particular assumption of the returns distribution with maximum entropy, i.e., a uniform prior in the returns space--- a variation of Uniswap v3 with variable tick spacing $\delta$ achieves an approximation error that matches (up to a constant multiplicative factor) the lower bound in \Cref{thm:lower}.
The precise formulation follows.

\begin{theorem}
\label{thm:v3upper}
For every $\eps > 0$, there exists a Uniswap v3-like exchange mechanism $\calG$ with $n=\Omic(1/\eps^p)$ ticks at prices $\pmin (1+\delta)^i$ for $i\in\{0,1,\dots,n\}$ where $\log(1+\delta) = \eps^p \log(\pmax/\pmin)$, that attains approximation error according to \Cref{eq:err} with a normalized weight function $w(p)$ which assigns measure at most $\Omic(1/n)$ to each of the intervals defined by these ticks, of
\[
\text{err}(\calG) \le \Omic(\eps \cdot (\fmax - \fmin))
\,.
\]
\end{theorem}

The detailed proof of \Cref{thm:v3upper} is relegated to \Cref{subsec:proofv3upper}.

\section{Proofs}
\label{sec:proofs}

\subsection{Proof of \Cref{thm:upper}}
\label{subsec:proofupper}

Let $\eps>0$, and a normalized weight function $w \colon [\pmin,\pmax]\to\Rp$ such that $\int_\pmin^\pmax
w(p)\, dp=1$.
Then, since $w(p)\ge 0\ \forall p\in[\pmin,\pmax]$, split the interval $[\pmin,\pmax]$ into $n=1/\eps^p$ equal measure (according to the weight function) sub-intervals $[t_i, t_{i+1}],\ \forall i\in\{1,2,\dots,n\}$, i.e., such that $\int_{t_i}^{t_{i+1}} w(p) dp = \frac{1}{n}$.
Define the limit order book (LOB) exchange mechanism $\calG=\cone(\calG)$ as the conical hull of the following set of basis functions: each basis function represents a limit order at each price point $t_i$ above,
i.e., the basis function is a unit step function dropping from 1 to 0 at price $t_i$.
The exchange complexity of this $\calG$ is therefore $1/\eps^p$.

Consider any $f\in\calF$, and define the following $g_f\in\cone(\calG)$ that will ``approximate'' this $f$:
\begin{equation}
\label{eq:defg}
\forall p\in(t_i, t_{i+1}),\ g_f(p) = \frac{f(t_i) + f(t_{i+1})}{2}
\,.
\end{equation}
It is true that this $g_f\in\cone(\calG)$, because $g_f$ is piecewise constant, with function value drops occurring only at the prices $t_i$ (see \Cref{subfig:lob} for an example representation).

We have that
\[
\forall p\in(t_i, t_{i+1}),
\
|f(p)-g_f(p)| \leq \frac{f(t_i) - f(t_{i+1})}{2}
\,,
\]
since $f$ is non-increasing, and by the definition of $g_f$ in \Cref{eq:defg}.

Hence, we obtain the desired result:
\begin{align*}
\text{err}(\calG)
=
\sup_{f\in\calF} \left\{ \inf_{g\in\cone(\calG)} d(f,g) \right\}
&\le
\sup_{f\in\calF} \left( \sum_{i=1}^{n} \int_{t_i}^{t_{i+1}} w(s) \left| f(s)-g_f(s) \right|^p ds \right)^{1/p}
\\ &\le
\sup_{f\in\calF} \left( \sum_{i=1}^{n} \int_{t_i}^{t_{i+1}} w(s) \left( \frac{f(t_i) - f(t_{i+1})}{2} \right)^p ds \right)^{1/p}
\\ &=
\frac{1}{2n^{1/p}} \sup_{f\in\calF} \left( \sum_{i=1}^{n} \left[ f(t_i) - f(t_{i+1}) \right]^p \right)^{1/p}
\\ &\le
\frac{1}{2n^{1/p}} \sup_{f\in\calF} \sum_{i=1}^{n} \left[ f(t_i) - f(t_{i+1}) \right]
\\ &\le
\eps \cdot \frac{\fmax - \fmin}{2}
\,,
\end{align*}
where the second-to-last inequality follows from the inequality between $\ell_1$ and $\ell_p$ norms in the function space.

\subsection{Proof of \Cref{thm:lower}}
\label{subsec:prooflower}

Let $\eps>0$, and a normalized weight function $w \colon [\pmin,\pmax] \to \Rp$ such that
$\int_\pmin^\pmax w(p) \, dp=1$.
Similarly to the upper bound, but with double the amount of intervals, split the interval $[\pmin,\pmax]$ into $2(n+2)$ (where $n=1/\eps^p$) equal measure (according to the weight function) sub-intervals $[t_i, t_{i+1}],\ \forall i\in\{1,2,\dots,2n+4\}$, i.e., such that $\int_{t_i}^{t_{i+1}} w(p) dp = \frac{1}{2(n+2)}$.
Now, consider any exchange mechanism $\calG$ with exchange complexity $\le \frac{1}{\eps^p}-1$, i.e., such that $\cone(\calG)$ is generated by $\le \frac{1}{\eps^p}-1$ basis functions; suppose without loss of generality that these are $g_1,g_2,\dots,g_{n-1} \in\cone(\calG)$.

\begin{lemma}
\label{lemma:lowerbasis}
For every basis function $g_i$ (where $i\in\{1,2,\dots,n-1\}$ as above), there exists at most one interval of the form $[t_{2l+1}, t_{2l+3}]$ for some $l\in\{1,\dots,n\}$ (where $t$'s are defined as in the above paragraph) such that
\[
g_i(t_{2l+1}) - g_i(t_{2l+3})
>
\frac{g_i(t_3) - g_i(t_{2n+3})}{2}
\,.
\]
\end{lemma}
\begin{proof}
Let $g_i$ be any basis function.
Assume that the lemma's hypothesis is not true, i.e., there exist at least two intervals $[t_{2l+1}, t_{2l+3}]$ and $[t_{2m+1}, t_{2m+3}]$ for some $l, m$ such that the lemma's equation holds for each of these intervals.
But since $g_i$ is non-increasing, this would necessitate that
\begin{align*}
g_i(t_3) - g_i(t_{2n+3}) &\ge \big[g_i(t_{2l+1}) - g_i(t_{2l+3})\big] + \big[g_i(t_{2m+1}) - g_i(t_{2m+3})\big] \\ &> g_i(t_3) - g_i(t_{2n+3})
\,,
\end{align*}
which completes the proof by contradiction.
\end{proof}

From \Cref{lemma:lowerbasis} and the pigeonhole principle (there exist $n$ odd-indexed intervals of the form $[t_{2l+1}, t_{2l+3}]$ for some $l\in\{1,\dots,n\}$, but only $n-1$ basis functions), we get that there exist at least one interval of the form $[t_{2l+1}, t_{2l+3}]$ (for some $l\in\{1,\dots,n\}$) such that for all $i\in\{1,2,\dots,n-1\}$,
\[
g_i(t_{2l+1}) - g_i(t_{2l+3})
\le
\frac{g_i(t_3) - g_i(t_{2n+3})}{2}
\,,
\]
and because $\cone(\calG)$ is finitely generated, it holds that for all $g\in\cone(\calG)$,
\begin{equation}
\label{eq:gsmalljump}
g(t_{2l+1}) - g(t_{2l+3})
\le
\frac{g(t_3) - g(t_{2n+3})}{2}
\,.
\end{equation}
Note that the interval is not the leftmost $[t_1, t_3]$ or the rightmost $[t_{2n+3}, t_{2n+5}]$ interval.

Consider the following specific $f_a\in\calF$:
\[
f_a(p) =
\begin{cases}
\fmax\,, &\text{for } \pmin \le p < t_{2l+2} \\
\fmin\,, &\text{for } t_{2l+2} \le p \le \pmax
\end{cases}
\,.
\]

Consider any $g\in\cone(\calG)$. We distinguish a few cases for the extreme values of $g$ outside of the outermost odd-indexed intervals, i.e., $g(t_3)$ and $g(t_{2n+3})$:
\begin{itemize}
    \item If $g(t_3) \ge \fmax + \frac{\fmax-\fmin}{4}$, then
    \[
    \int_{t_1}^{t_3} w(s) \left| f_a(s)-g(s) \right|^p ds
    \ge
    \frac{(\fmax-\fmin)^p}{(n+2)\cdot 4^p}
    \,.
    \]
    \item If $g(t_{2n+3}) \le \fmin - \frac{\fmax-\fmin}{4}$, then
    \[
    \int_{t_{2n+3}}^{t_{2n+5}} w(s) \left| f_a(s)-g(s) \right|^p ds
    \ge
    \frac{(\fmax-\fmin)^p}{(n+2)\cdot 4^p}
    \,.
    \]
    \item Otherwise, we have that $g(t_3) - g(t_{2n+3}) < \frac{3}{2} \left( \fmax-\fmin \right)$. We now distinguish 3 sub-cases:
    
    \begin{itemize}
        \item If $g(t_{2l+1}) \ge \fmax$, then $g(t_{2l+2}) \ge g(t_{2l+3}) \ge \frac{\fmax+3\fmin}{4}$ by \Cref{eq:gsmalljump}, thus
        \[
        \int_{t_{2l+2}}^{t_{2l+3}} w(s) \left| f_a(s)-g(s) \right|^p ds
        \ge
        \frac{(\fmax-\fmin)^p}{(n+2)\cdot 2^{1+2p}}
        \,.
        \]
        \item If $g(t_{2l+3}) \le \fmin$, then $g(t_{2l+2}) \le g(t_{2l+1}) \le \frac{3\fmax+\fmin}{4}$ by \Cref{eq:gsmalljump}, thus
        \[
        \int_{t_{2l+1}}^{t_{2l+2}} w(s) \left| f_a(s)-g(s) \right|^p ds
        \ge
        \frac{(\fmax-\fmin)^p}{(n+2)\cdot 2^{1+2p}}
        \,.
        \]
        \item Otherwise, for some $\delta_1, \delta_2 > 0$ we have that $\fmin < \fmin + \delta_2 = g(t_{2l+3}) \le g(t_{2l+1}) = \fmax - \delta_1 < \fmax$; then by \Cref{eq:gsmalljump} we get $\delta_1 + \delta_2 \ge \frac{\fmax-\fmin}{4}$, therefore
        \[
        \int_{t_{2l+1}}^{t_{2l+3}} w(s) \left| f_a(s)-g(s) \right|^p ds
        \ge
        \frac{\delta_1^p + \delta_2^p}{2(n+2)}
        \ge
        \frac{(\delta_1 + \delta_2)^p}{(n+2)\cdot 2^p}
        \ge
        \frac{(\fmax-\fmin)^p}{(n+2)\cdot 8^p}
        \,,
        \]
        where the second-to-last inequality follows from H\"older's inequality.
    \end{itemize}
\end{itemize}

Hence, we obtain the desired result:
\begin{align*}
\text{err}(\calG)
=
\sup_{f\in\calF} \left\{ \inf_{g\in\cone(\calG)} d(f,g) \right\}
&\ge \inf_{g\in\cone(\calG)} \left(
\int_\pmin^\pmax w(s) \left| f_a(s)-g(s) \right|^p ds
\right)^{1/p}
\\ &\ge
\eps \cdot \Omeg(\fmax - \fmin)
\,.
\end{align*}

\subsection{Proof of \Cref{thm:v3upper}}
\label{subsec:proofv3upper}

Let $\eps>0$, and consider ticks $t_i=\pmin (1+\delta)^i$ for $i\in\{0,1,\dots,n\}$ where $\log(1+\delta) = \eps^p \log(\pmax/\pmin)$, and $n=\log(\pmax/\pmin)/\log(1+\delta)$, so that $t_0 = \pmin$ and $t_n = \pmax$.
Consider the normalized weight function $w \colon [\pmin,\pmax] \to \Rp$ such that
$\int_\pmin^\pmax w(p) \, dp=1$, with the property that for some constant $C>0$, $\forall
i\in\{0,1,\dots,n-1\},\ \int_{t_i}^{t_{i+1}} w(p) \, dp \le \frac{C^p}{n}$.
Our Uniswap v3-like exchange mechanism $\calG=\cone(\calG)$ is described with the following $n+1$ basis functions: one basis function for each of the intervals $[t_i, t_{i+1}]$ for $i\in\{0,1,\dots,n-1\}$ defined by
\[
g_i(p) =
\begin{cases}
\frac{1}{\sqrt{t_i}} - \frac{1}{\sqrt{t_{i+1}}} \,, &\text{for } \pmin \le p\le t_i \\
\frac{1}{\sqrt{p}} - \frac{1}{\sqrt{t_{i+1}}} \,, &\text{for } t_i \le p \le t_{i+1} \\
0 \,, &\text{for } t_{i+1} \le p \le \pmax
\end{cases}
\,,
\]
along with the additional basis function $g_n(p)$ that is everywhere 1\footnote{Note that this additional basis function is always necessary whenever $\fmin\ne 0$ to obtain an \emph{arbitrarily good} approximation of any curve, due to the construction of the Uniswap curves to end at exactly 0 at the end of each interval.}.

Consider any $f\in\calF$, and define the following $g_f\in\cone(\calG)$ that will ``approximate'' this $f$:
\[
g_f(p) = f(\pmax) g_n(p) + \sum_{i=0}^{n-1} \frac{f(t_i) - f(t_{i+1})}{\frac{1}{\sqrt{t_i}} - \frac{1}{\sqrt{t_{i+1}}}} g_i(p)
\,.
\]

Then, it holds that
\[
\forall p\in(t_i, t_{i+1}),
\
|f(p)-g_f(p)| \leq f(t_i) - f(t_{i+1})
\,.
\]

Hence, we obtain the stated result by a similar argument to that of \Cref{subsec:proofupper}.

\section*{Acknowledgments} 
We would like to thank Neel Tiruviluamala for helpful comments on the proof of our lower bound.
We also thank anonymous reviewers of the Financial Cryptography and Data Security conference for useful suggestions.
The first author is supported in part by NSF awards CNS-2212745, CCF-2212233, DMS-2134059, and CCF-1763970.
The second author is supported by the Briger Family Digital Finance Lab at Columbia Business School.
The third author's research at Columbia University is supported in part by NSF awards CNS-2212745, and CCF-2006737.

\section*{Disclosures}
The first author is a Research Fellow with automated market making protocols, including ones mentioned in this work.
The second author is an advisor to fintech companies.
The third author is Head of Research at a16z crypto, which reviewed a draft of this article for compliance prior to publication and is an investor in various decentralized finance projects, including Uniswap, as well as in the crypto ecosystem more broadly (for general a16z disclosures, see \url{https://www.a16z.com/disclosures/}).

\printbibliography

\appendix

\section{Deferred Proofs of \Cref{subsec:model_props}}
\label{app:proofs_model_props}

\begin{proof}[Price discovery]
Assume that the current price of the exchange is $p_0 \ne p$. Suppose that an external market participant comes to the exchange and is willing to trade to some price $p_1$, and then uses the external market to trade back. We prove that the maximum profits will be obtained at $p_1=p$; therefore, if the trader does not maximize their profits, other external market participants will continue to have an incentive to trade until the price of the exchange is $p$ and the conclusion follows.

Due to \Cref{eq:trade_numeraire}, the external market participant's optimization problem for their profit is:
\begin{equation*}
\max_{p_1\in\Rp} p (g(p_0) - g(p_1)) + \int_{p_0}^{p_1} pdg(p)
= \max_{p_1\in\Rp} (p_1 - p) g(p_1) - \int_{0}^{p_1} g(p)dp
\end{equation*}

First-order conditions then prove that the optimum is attained at $p_1=p$.
\end{proof}

\begin{proof}[Budget balance]
Assume that the current price of the exchange is $p_0$.
First, we note that liquidity additions and removals, due to the linear nature of the aggregate demand curves and the num\'eraire contributed/removed by \Cref{eq:lp_numeraire_deposit} with respect to the curves $g_i(p)$, do not affect the rest of the joint pool, i.e., if the amount of num\'eraire was non-negative before the operation, so it is after it.
Trading is the only action which is yet unclear how it affects the amount of num\'eraire in the pool.
In aggregate, the joint pool contains a quantity $g(p_0)$ of risky asset, and in num\'eraire by \Cref{eq:lp_numeraire_deposit}:
\[
\sum_{i=1}^n -\int_0^{p_0} pdg_i(p)
= -\int_0^{p_0} pdg(p)
\ge 0
\,,
\]
because $g$ is non-increasing (as the sum of non-increasing functions) and $p_0\ge 0$.
Suppose that a trader comes and moves the pool price to $p_1$. The new amount of num\'eraire contained in the pool by the above equation and \Cref{eq:trade_numeraire} is
\[
-\int_0^{p_0} pdg(p) - \int_{p_0}^{p_1} pdg(p)
= -\int_0^{p_1} pdg(p)
\ge 0
\,,
\]
thereby completing our argument.
\end{proof}

\end{document}